\documentclass[letterpaper, 10 pt, conference]{ieeeconf}  

\IEEEoverridecommandlockouts                              
\usepackage{amsfonts, amsmath, bm}
\usepackage{algorithmic, graphicx}
\usepackage{epstopdf}
\usepackage{float}
\usepackage{subcaption}
\usepackage{color}
\usepackage{xcolor}
\newcommand{\edit}[1]{{\color{black}{#1}}}

\title{\LARGE \bf
Convergence of Iterative Quadratic Programming for Robust Fixed-Endpoint Transfer of Bilinear Systems
} 
\author{Luke S. Baker$^1$, Andre Luiz P. de Lima$^2$, Anatoly Zlotnik$^3$, Jr-Shin Li$^2$, and Michael J. Martin$^4$
\thanks{This project was supported by the LDRD program and the Center for Nonlinear Studies at Los Alamos National Laboratory.  Research conducted at Los Alamos National Laboratory is done under the auspices of the National Nuclear Security Administration of the U.S. Department of Energy under Contract No. 89233218CNA000001. Report No. LA-UR-24-22511.}
\thanks{$^1$lsbaker@lanl.gov, \,\,  Center for Nonlinear Studies, Theoretical Division, Los Alamos National Laboratory, Los Alamos, NM, USA 87545}
\thanks{$^2$\{a.delima,jsli\}@wustl.edu, \,\, Electrical \& Systems Engineering Department, Washington University in St. Louis, St. Louis,
MO, USA, 63130}
\thanks{$^3$azlotnik@lanl.gov, \,\, Applied Mathematics \& Plasma Physics, Los Alamos National Laboratory, Los Alamos, NM, USA 87545}
\thanks{$^4$mmartin@lanl.gov, \,\, Quantum Technologies Team, Los Alamos National Laboratory, Los Alamos, NM, USA 87545}}

\newtheorem{thm}{Theorem}

\newtheorem{prop}[thm]{Proposition}

\begin{document}

\maketitle
\thispagestyle{empty}
\pagestyle{empty}

\begin{abstract}
We present a computational method for open-loop minimum-norm control synthesis for fixed-endpoint transfer of bilinear ensemble systems that are indexed by two continuously varying parameters.  We suppose that one ensemble parameter scales the homogeneous, linear part of the dynamics, and the second parameter scales the effect of the applied control inputs on the inhomogeneous, bilinear dynamics.  This class of dynamical systems is motivated by robust quantum control pulse synthesis, where the ensemble parameters correspond to uncertainty in the free Hamiltonian and inhomogeneity in the control Hamiltonian, respectively.  Our computational method is based on polynomial approximation of the ensemble state in parameter space and discretization of the evolution equations in the time domain using a product of matrix exponentials corresponding to zero-order hold controls over the time intervals.  The dynamics are successively linearized about control and trajectory iterates to formulate a sequence of quadratic programs for computing perturbations to the control that successively improve the objective until the iteration converges.  We use a two-stage computation to first ensure transfer to the desired terminal state, and then minimize the norm of the control function.  The method is demonstrated for the canonical uniform transfer problem for the Bloch system that appears in nuclear magnetic resonance, as well as the matter-wave splitting problem for the Raman-Nath system that appears in ultra-cold atom interferometry.
\end{abstract}

\section{Introduction} \label{sec:intro}



The synthesis of open-loop optimal controls for bilinear dynamical systems has been studied for decades \cite{aganovic1994successive}.  The bilinear dynamics in this setting are characterized by a system of linear ordinary differential equations (ODEs), where some coefficients are time-varying control functions for which minimal energy candidates are desired.  The so-called bilinear-quadratic problems that result have been addressed by iterative feedback control methods \cite{hofer1988iterative}, and the approach has been extended to open loop controls for fixed-endpoint state transfers \cite{wang2017fixed}.  In practice, it is of interest to synthesize controls that are robust or insensitive to uncertainty or variation in system parameters, and this leads to infinite-dimensional systems \cite{beauchard2010controllability}.  Because state feedback is impractical or unavailable in this setting, open-loop controls are sought.  Control robustness is typically understood as a uniform state transfer effect on all dynamical units in an ensemble or collection of structurally-similar systems indexed by parameters varying on a compact set \cite{wang2017fixed, zhang2019analyzing}.  The desired control solution will steer the entire ensemble from the initial state to within an allowable distance from the desired target state.  Rigorous definitions and conditions under which ensemble controllability is assured have been established for ensemble systems that have certain bilinear structures \cite{li2009ensemble, zhang2019analyzing, li2020separating}.


Interest in control synthesis methods for fixed-endpoint transfers of bilinear ensembles has been driven over the past decades by problems related to quantum control applications \cite{li2011optimal}. Several approaches to solve the associated optimal control problem (OCP) employ fixed-point iteration directly either using linearization and approximation by Freholm operators \cite{zlotnik2012iterative} or by solving the quadratic-bilinear Riccati problem \cite{wang2017fixed}.  A variety of methods have been proposed to approximate infinite-dimensional ensembles using finite-dimensional representations, which typically involves spectral approximation \cite{ruths2011multidimensional}.  A promising recent concept involves so-called polynomial moments \cite{narayanan2020moment}, in which the system dynamics are represented in ensemble space over a basis of orthogonal functions.  Moment dynamical system representations were recently used for control synthesis in bilinear systems that appear in quantum applications \cite{ning2022nmr}.  


There are significant trade-offs between accurate representation of system dynamics and scale of the computational representation when implementing optimal controls using polynomial moment dynamics.  Depending on the truncation of the polynomial order, error tolerances, and time horizon, the resulting nonlinear program that discretizes the OCP can require complicated representations and very large numbers of optimization variables \cite{delima2023optimal}.  A promising recent approach to optimal control syntheses for nonlinear systems subject to constraints is to apply iterative quadratic programming to a sequence of linear approximations to the dynamics that are locally updated at each iteration \cite{vu2020iterative}. 


In this study, we develop a computational method for open-loop minimum-norm control synthesis for fixed-endpoint transfer of a class of bilinear ensemble systems that are indexed by two continuously varying parameters, subject to constraints on the controls.  We suppose that one ensemble parameter scales the homogeneous, linear part of the dynamics, and the second parameter scales the effect of the applied control inputs on the inhomogeneous, bilinear dynamics.  The class of systems with this structure can be applied to model a broad range of phenomena in the control of quantum and robotic systems \cite{altafini2012modeling,becker2012approximate}.  We examine in particular aspects of the linearization and discretization that promote computational scalability of the numerical algorithm.  We show that the order in which linearization and discretization are applied to the bilinear system can result in different approximations of the ensemble trajectory, so these operations are not in general commutative.  These results are in agreement with prior studies on dynamical systems \cite{grammont2014commute, breda2016pseudospectral}, which show that such commutation and approximation quality depend on the structure of the system and the discretization method. In addition to characterizing the discretization that leads to the best approximation, we also prove that linearization and discretization operations commute in the limit of numerical endpoint quadrature.  Finally, we demonstrate the generality of the method through computational experiments that involve two bilinear systems that arise in quantum control.


The rest of this paper is organized as follows.  Minimal energy control of a collection of bilinear dynamical systems is formulated in Section \ref{sec:robust_control}, and the reduction to a finite-dimensional system using the method of moments is presented there as well. Section \ref{sec:linearization_discretization} provides details of linearization and discretization of the reduced dynamical system.  Section \ref{sec:sqp} presents the iterative quadratic program used to compute the minimal energy control function.  Results of the control design are demonstrated in Section \ref{sec:examples} for numerical applications in nuclear magnetic resonance and ultra-cold atom interferometry.  Concluding remarks and an outlook for further development of the control algorithm are presented in Section \ref{sec:conc}.

\section{Robust Optimal State Transfer for a Continuum of Bilinear Systems} \label{sec:robust_control}

We formulate an OCP for a class of bilinear systems with dynamics that are affected by two parameters that vary over compact intervals.

\subsection{Bilinear Ensemble System} \label{sec:bilinear_system}

We consider an uncountable collection of structurally-identical bilinear dynamical systems of the form
\begin{equation} \label{eq:bilinear_ensemble}
        \dot X(t;\alpha,\beta)=\alpha \mathcal A X(t;\alpha,\beta)+\beta\sum_{i=1}^m U_i(t)\mathcal B_i  X(t;\alpha,\beta),
\end{equation}
where $U_i(t)\in [U_{\min}, U_{\max}]$ ($i=1,\dots,m$) represent control input functions and $X(t;\alpha,\beta)\in \mathbb{R}^n$ represents the state of the ensemble of bilinear systems indexed by parameters $\alpha\in [\alpha_{\min},\alpha_{\max}]$ and $\beta\in [\beta_{\min},\beta_{\max}]$ that affect the evolution of individual dynamical units in the ensemble.  The constant matrix $\mathcal A \in\mathbb{R}^{n\times n}$ characterizes the homogeneous part of the state dynamics, and each $\mathcal B_i\in\mathbb{R}^{n\times n}$ characterizes the influence of input $U_i$ on the state evolution for each $i=1,\dots,m$.  We refer to the parameterized collection of bilinear systems and the associated collection of indexed states as the ensemble system and the ensemble state, respectively.  For each fixed pair of parameters $\alpha$ and 
$\beta$, the above equation describes the time-evolution of the associated state $X(t;\alpha,\beta)$ as it evolves under the influence of the control function $U_i(t)$ ($i=1,\dots,m$).  The parameters $\alpha$ and $\beta$ are used to represent intrinsic system modeling uncertainty and inhomogeneity in applied control actuations, respectively.  This class of bilinear ensemble systems can be used to broadly represent a variety of quantum dynamical phenomena and associated control systems \cite{altafini2012modeling}.

\subsection{Optimal Control Problem} \label{sec:ocp}

Given a specified finite time $T$, we seek a single open-loop control solution of minimal energy that steers the ensemble state from uniform initial state $X_0\in \mathbb{R}^n$ to uniform target terminal state $X_T\in \mathbb{R}^n$ during the time interval $[0,T]$.  These endpoint conditions take the form
\begin{subequations}\label{eq:endpoint_equalities}
\begin{align} 
\!\!\!\!X(0;\alpha,\beta)&\!=\!X_0, \, \forall\alpha \!\in\! [\alpha_{\min},\alpha_{\max}], \, \forall\beta \!\in\! [\beta_{\min},\beta_{\max}], \! \label{eq:endpoint_equalities_init}\\
\!\!\!\!X(T;\alpha,\beta)&\!=\!X_T, \,  \forall\alpha \!\in\! [\alpha_{\min},\alpha_{\max}], \,  \forall\beta \!\in\! [\beta_{\min},\beta_{\max}]. \!\! \label{eq:endpoint_equalities_term}
\end{align}
\end{subequations}
\edit{Although the initial and target states are assumed to be independent of the parameters $\alpha$ and $\beta$, the setting may be extended to selective excitation in which distinct initial and target states could be associated to disjoint subsets of the parameter space \cite{pauly1991parameter}, i.e. $X_0$ and $X_T$ could depend on $\alpha$ and $\beta$ \cite{ning2022nmr}}.  We further suppose that control inputs are constrained by application requirements for all $i=1,\dots,m$ according to the inequalities
\begin{subequations} \label{eq:control_inequalities}
\begin{align} 
U_{\min}&\le U_i(t)\le U_{\max}, \,\, &\forall t\in[0,T], \\ 
\Delta U_{\min}&\le \dot U_i(t)\le \Delta U_{\max},  \,\, &\forall t\in[0,T],
\end{align}
\end{subequations}
where the amplitude constraint bound values $U_{\min}$ and $U_{\max}$  and the derivative limits $\Delta U_{\min}$ and $\Delta U_{\max}$ are problem parameters. The objective function for the variational minimization is the energy of the applied control defined by
\begin{equation} \label{eq:ocp_objective}
\min_{\{U_i(t)\}_{i=1}^m}    \sum_{i=1}^m\int_0^T\|U_i(t)\|^2dt.
\end{equation}
The notation $\|x\|^2=x'x$ indicates the squared Euclidean norm of a vector $x$, where $x'$ denotes the transpose of $x$.   The objective in equation \eqref{eq:ocp_objective} is minimized subject to the dynamic constraints  \eqref{eq:bilinear_ensemble}, the initial and terminal conditions in equations \eqref{eq:endpoint_equalities}, and the control amplitude and derivative constraints in equations \eqref{eq:control_inequalities}.  In our computational implementation, we explicitly enforce the initial state condition $X(0;\alpha,\beta)=X_0$ as defined in equation \eqref{eq:endpoint_equalities_init}, and relax the terminal state condition \eqref{eq:endpoint_equalities_term} to the inequality  \edit{
\begin{align} \label{eq:endpoint_equalities_term_relaxed}
\|P\left(X(T;\alpha,\beta)-X_T\right)\|\le \epsilon, 
\end{align}
where $\epsilon$ is a positive error tolerance.  The above inequality introduces a matrix $P$, which may represent the identity or a projection matrix onto a lower-dimensional subspace. This degree of flexibility provides a framework to maximize expectation values or probabilities for problems governed by quantum mechanical systems in which the components of the state vectors are complex-valued time-varying functions.}


\subsection{Spectral Approximation in Parameter Space by Polynomial Moments} \label{sec:moments}

We develop a numerical approximation method to represent the uncountable parameter space using a finite-dimensional polynomial moment expansion.  Rather than direct sampling of the parameter space, we consider a superposition of the ensemble state onto orthogonal basis functions over the parameter domain $[\alpha_{\min},\alpha_{\max}] \times [\beta_{\min},\beta_{\max}]$, and then truncate the series to obtain a finite approximation.  We employ Legendre polynomials as the orthogonal basis, following a foundational study on ensemble dynamics \cite{zeng2016moment}.  First, we transform the two parameter intervals $[\alpha_{\min},\alpha_{\max}]$ and $[\beta_{\min},\beta_{\max}]$ to the interval $[-1,1]$ on which Legendre polynomials are defined.  The transformations are defined as
\begin{equation} \label{eq:domain_transformations}
   \alpha(a)=\underline \alpha a+\overline \alpha, 
   \qquad  \beta(b)=\underline \beta b+\overline \beta, 
\end{equation}
in which we use the notation $\overline \gamma=(\gamma_{\max}+ \gamma_{\min})/2$ and $\underline \gamma=(\gamma_{\max}- \gamma_{\min})/2$ where $\gamma$ represents one of $\alpha$ or $\beta$.
Observe that $\alpha(-1)=\alpha_{\min}$, $\alpha(1)=\alpha_{\max}$, $\beta(-1)=\beta_{\min}$, and $\beta(1)=\beta_{\max}$.  Define the normalized Legendre polynomial of degree $k$ as a function of the variable $\gamma\in [-1,1]$ by
\begin{eqnarray} \label{eq:normalized_Legendre}
    L_k(\gamma)=\frac{\sqrt{2k+1}}{\sqrt{2}2^{k}k!}{\frac {d^{k}}{d\gamma^{k}}}(\gamma^{2}-1)^{k}.
\end{eqnarray} 
The functions in equation \eqref{eq:normalized_Legendre} satisfy the recurrence relation
\begin{equation} \label{eq:Legendre_recurrence}
    \gamma L_k(\gamma)=c_{k-1}L_{k-1}(\gamma)+c_kL_{k+1}(\gamma),
\end{equation}
where $c_k=(k+1)/{\sqrt{(2k+3)(2k+1)}}$. We assume that $X(t,\alpha(a),\beta(b))$ is square-integrable over $a,b\in [-1,1]$ for all $t\in [0,T]$, and that $\dot{X}(t,\alpha(a),\beta(b))$ is continuous for $t\in [0,T]$ and $a,b\in [-1,1]$. Using the completeness and orthonormality of the normalized Legendre polynomials on the interval $[-1,1]$, we expand the ensemble state as
\begin{equation}
    X(t;\alpha(a),\beta(b))=\sum_{p,q=0}^{\infty} x_{p,q}(t)L_p(a)L_q(b),
\end{equation}
where the expansion coefficients are 
\begin{equation}
    x_{p,q}(t)=\int_{-1}^1 \int_{-1}^{1}X(t;\alpha(a),\beta(b)) L_p(a)L_q(b) da db.
\end{equation}
By truncating the series, we obtain a numerically tractable approximation given by
\begin{equation} \label{eq:truncated_series}
    X(t;\alpha(a),\beta(b))\approx\sum_{p=0}^{N_{\alpha}}\sum_{q=0}^{N_{\beta}}x_{p,q}(t)L_p(a)L_q(b),
\end{equation}
\edit{where $N_{\alpha}$ and $N_{\beta}$ are the maximum degrees of the Legendre polynomials defined on the respective parameter intervals.}
By the dominated convergence theorem and the recurrence relation \eqref{eq:Legendre_recurrence}, the dynamics of the coefficients are shown to satisfy the system of differential equations
\begin{IEEEeqnarray}{lll}
   \dot x_{p,q}(t)
    &=& \mathcal A\left(c_{p-1}\underline{\alpha}x_{p-1,q}(t)+\overline{\alpha}x_{p,q}(t)+c_{p}\underline{\alpha}x_{p+1,q}(t)\right) \nonumber \\
    && +\sum_{i=1}^m U_i(t)\mathcal B_i\left( c_{q-1}\underline{\beta}x_{p,q-1}(t)+\overline{\beta}x_{p,q}(t)\right. \nonumber \\
    && \qquad \qquad  \quad \left. +c_{q}\underline{\beta}x_{p,q+1}(t) \right), \quad \forall p,\,q. \label{eq:coefficient_dynamics}  
\end{IEEEeqnarray}
All terms of the form $x_{p,N_{\beta}+1}$ and $x_{N_{\alpha}+1,q}$ for $p=0,\dots,N_{\alpha}$ and $q=0,\dots,N_{\beta}$ are removed from the expressions in equation \eqref{eq:coefficient_dynamics}.  The initial and desired target states of the ensemble correspond uniquely to initial and target states of the expansion coefficients.  In particular, $x_{0,0}(0)=2X_0$ and $x_{0,0}(T)=2X_T$, whereas $x_{p,q}(0)$ and $x_{p,q}(T)$ are $n$-dimensional zero vectors for all $p,q \not=0$ because of the orthogonality of the Legendre polynomials and the independence of the initial and target states from the ensemble parameters. 

The above procedure reduces an uncountable collection \eqref{eq:bilinear_ensemble} of bilinear systems to an approximate finite-dimensional system \eqref{eq:coefficient_dynamics} of representative coefficients.  We can concatenate the dynamics of the ensemble state as represented by the coefficients by defining $x=[x_{0,0}',\dots, x_{0,N_{\beta}}',\dots,x_{N_{\alpha},0}',\dots,x_{N_{\alpha},N_{\beta}}']'$ and the $(N_{\gamma}+1)\times (N_{\gamma}+1)$ tri-diagonal symmetric matrices
\begin{IEEEeqnarray*}{lll}
    C_{\gamma}=\begin{bmatrix}
        \overline \gamma & c_0\underline \gamma &  &  & &    \\
        c_0\underline \gamma & \overline \gamma & c_1\underline \gamma &  & &  \\
         & c_1\underline \gamma & \overline \gamma &  & &   \\
         & & & \ddots & &  \\
         & & & & \overline \gamma &  c_{N_{\gamma}-1}\underline \gamma \\
         & & & & c_{N_{\gamma}-1}\underline \gamma &  \overline \gamma
    \end{bmatrix}
\end{IEEEeqnarray*}
for $\gamma=\alpha,\, \beta$. We also define the $n(N_{\alpha}+1)(N_{\beta}+1)$-dimensional vectors $x_0=[2X_0',0,\dots,0]'$ and $x_T=[2X_T',0,\dots,0]'$ to represent the initial and target states in the truncated polynomial coefficient space. With these definitions, the dynamics in terms of the coefficients as stated in equation \eqref{eq:coefficient_dynamics} may be written as
\begin{equation} \label{eq:bilinear_ODE}
        \dot x(t)= A x(t)+\sum_{i=1}^m U_i(t)B_i x(t),
\end{equation}
where the $n(N_{\alpha}+1)(N_{\beta}+1)\times n(N_{\alpha}+1)(N_{\beta}+1)$ matrices are defined by
\begin{eqnarray} \label{eq:bilinear_ODE_A_B}
    A= C_{\alpha} \otimes I_{N_{\beta}+1} \otimes \mathcal A, \qquad B_i=  I_{N_{\alpha}+1} \otimes C_{\beta} \otimes \mathcal B_i.
\end{eqnarray}
Here, $I_{M}$ represents the $M\times M$ identity matrix and $C\otimes D$ represents the Kronecker product of matrices $C$ and $D$.
Our subsequent exposition is done for the finite-dimensional system in equations \eqref{eq:bilinear_ODE}-\eqref{eq:bilinear_ODE_A_B}.  The initial and desired target states of this system are equal to $x_0$ and $x_T$, respectively, as defined above.

\section{Linearization and Time-Discretization} \label{sec:linearization_discretization}

The iterative optimization algorithm that we develop to solve the OCP defined in Section \ref{sec:ocp} requires linearization of the bilinear system dynamic constraints \eqref{eq:bilinear_ensemble} and a discrete-time representation.  In this section, we detail the linearization and discretization of the bilinear system in equations \eqref{eq:bilinear_ODE}-\eqref{eq:bilinear_ODE_A_B}.  The effect of the order in which linearization and discretization are applied to a dynamical system has been investigated and is generally found to be dependent on the system structure and the discretization method \cite{grammont2014commute, breda2016pseudospectral}. One of the results presented in this section verifies that the order in which linearization and exact discretization are performed gives rise to different expressions for the discrete-time linear approximation of the state trajectory.  Therefore, these operations do not commute, in general, when applied to a bilinear system of form \eqref{eq:bilinear_ensemble}.  However, we prove that these operations commute in an approximate sense and converge with finer discretization.  For both orderings, we consider a zero-order hold framework in which control variables are piece-wise constant over each specified time interval.

\subsection{Discretization Followed by Linearization} \label{sec:disc_then_lin}

The time interval $[0,T]$ is discretized into \edit{$K+1$ sampling times $t_0=0,\dots,t_K=T$.}  Under the assumption of zero-order hold, the bilinear system in equation \eqref{eq:bilinear_ODE} is equivalent to a linear time-invariant system over each sub-interval $[t_{k},t_{k+1}]$.  Thus the transition from $x(t_k)$ to $x(t_{k+1})$ is given by the matrix exponential expression
\begin{equation} \label{eq:transition}
    x(t_{k+1})=e^{\Delta t_k\left(A+\sum_{i=1}^m U_i(t_k)B_i\right)} x(t_k),
\end{equation}
where $\Delta t_k=t_{k+1}-t_k$.

Suppose that $\{\overline U_i(t_k)\}$ for $k=0,\ldots,K-1$ and $i=1,\ldots,m$ denotes a collection of nominal piece-wise constant controls used to advance a nominal state trajectory $\{\overline x(t_k)\}$ according to equation \eqref{eq:transition}.  The nominal state is defined to satisfy the initial condition $\overline x(0)=x_0$. Consider slightly perturbed piece-wise constant control inputs $\delta u_i(t_k)$ and the associated perturbed state $\delta x(t_k)$ of the bilinear system, so that $U_i(t_k)=\overline U_i(t_k)+\delta u_i(t_k)$ and $x(t_k)=\overline x(t_k)+\delta x(t_k)$. By regulating the norm of the perturbed control vector to be sufficiently small, as defined subsequently, we may consider the linear system approximation about the nominal control and state variables.

We proceed to linearize the discrete transition in equation \eqref{eq:transition}.  The matrix exponential associated with the updated control input is written explicitly as
\begin{IEEEeqnarray}{lll}
  \sum_{j=0}^{\infty}\frac{\Delta t_k^{j}}{j!} \left( A+\sum_{i=1}^m (\overline U_i(t_k)+\delta u_i(t_k)) B_i\right)^{j}.
\end{IEEEeqnarray}
Linearizing the above representation about $\delta u_i(t_k)=0$, over all $i=1,\dots,m$, gives the expression
\begin{IEEEeqnarray}{lll}
 \!\!\!\!\!\!\!\! \sum_{j=1}^{\infty}\frac{\Delta t_k^{j}}{(j-1)!} \left( A+\sum_{i=1}^m \overline U_i(t_k) B_i\right)^{\!\!j-1} \!\!\!\!\!\cdot \!\left( \sum_{i=1}^m  \delta u_i(t_k) B_i\right) \nonumber \\
\!\!\! =\!\Delta t_k\text{exp}\!\left(\!\Delta t_kA\!+\!\Delta t_k\!\sum_{i=1}^m \! U_i(t_k)B_i\!\right)\!\!\cdot\!\!\sum_{i=1}^m \! \delta u_i(t_k) B_i. \!\!
\end{IEEEeqnarray}
By linearizing equation \eqref{eq:transition}, we obtain the approximate dynamics of the perturbation as
\begin{equation} \label{eq:discrete_lin}
\delta x(t_{k+1})=\bm A(t_k)\delta x(t_k) +\bm B(t_k)\delta u(t_k),
\end{equation}
where $\delta u(t_k)=[\delta u_1(t_k),\dots,\delta u_m(t_k)]'$ and
\begin{IEEEeqnarray}{lll}
    \bm A(t_k)&=&\text{exp}\left(\Delta t_kA+\Delta t_k\sum_{i=1}^m \overline U_i(t_k)B_i\right), \label{eq:A}\\
    \bm B(t_k)&=&\Delta t_k\bm A(t_k) [B_1\overline x(t_k),\dots, B_m \overline x(t_k)].\label{eq:B}
\end{IEEEeqnarray}

\subsection{Linearization Followed by Discretization}
  
Let us reconsider the continuous-time bilinear system in equation \eqref{eq:bilinear_ODE}.  Linearizing in continuous-time about $\overline U(t)$ and $\overline x(t)$ results in
\begin{equation} \label{eq:lin_sys}
    \Delta\dot{x}(t)=\overline A(t)\Delta x(t)+\overline B(t)\delta u(t), 
\end{equation}
with the time-varying state and control matrices defined by
\begin{subequations}
\begin{align}
    &\overline A(t)=A+\sum_{i=1}^m \overline U_i(t)B_i, \\
    &\overline B(t)=[B_1\overline x(t),\dots,B_m\overline x(t)].
\end{align}
\end{subequations}
As before, the nominal state satisfies the initial condition $\overline x(0)=x_0$.  Under the assumption of zero-order hold, the above state matrix $\overline A(t)$ is time-invariant for $t\in [t_k,t_{k+1})$.  The transition from $\Delta x_k$ to $\Delta x_{k+1}$ is therefore given by
\begin{IEEEeqnarray}{lll}
    \Delta x_{k+1}=\bm A_k\Delta x_k+\int_{t_k}^{t_{k+1}}e^{(t_{k+1}-\tau)\overline A_k} \overline B(\tau)  d\tau  \delta u_k,\quad\label{eq:lin_discrete}
\end{IEEEeqnarray}
in which we denote the evaluation of a variable at time $t=t_k$ with a subscript of index $k$ for simplicity of exposition.  For example,  $\Delta x_k=\Delta x(t_k)$ and $\bm A_k=\bm A(t_k)$.

Equations \eqref{eq:discrete_lin} and \eqref{eq:lin_discrete} indicate that linearization and discretization of the bilinear system are not commutative operations, in general, even though both of the discrete transitions are computed exactly with closed form matrix exponential expressions.
We note that other methods of discretization may in fact commute with linearization.  For example, regardless of whether or not the controls are piecewise constant, the Euler discretization and linearization are commutative operations on the bilinear system.  We have the following result.

\begin{prop}
Suppose that the integration in equation \eqref{eq:lin_discrete} is evaluated using the left-endpoint quadrature method. Then $\delta x_k$ in equation \eqref{eq:discrete_lin} is equal to $\Delta x_k$ in equation \eqref{eq:lin_discrete} for all $k=0,\dots, K$.
\end{prop}
\begin{proof}
Applying the left-endpoint method to the integration in equation \eqref{eq:lin_discrete} results in
    \begin{IEEEeqnarray}{lrl}
        \int_{t_k}^{t_{k+1}}e^{(t_{k+1}-\tau)\overline A_k} \overline B(\tau)  d\tau \approx \Delta t_k e^{\Delta t_k\overline A_k}\overline B(t_k). \label{eq:left_endpoint_transition}
    \end{IEEEeqnarray}
The expression on the right-hand side of equation \eqref{eq:left_endpoint_transition} is the definition of $\bm B_k$ in equation \eqref{eq:B}.  Therefore, from the hypothesis of the proposition, the state and control matrices in equations \eqref{eq:discrete_lin} and \eqref{eq:lin_discrete} are equivalent.  Because $\delta u_k$ is the same control function used in both equations \eqref{eq:discrete_lin} and \eqref{eq:lin_discrete}, we have
    \begin{IEEEeqnarray}{lrl}
\delta x_{k+1}-\Delta x_{k+1}=\bm A_k(\delta x_{k}-\Delta x_{k}) \label{eq:error_evolution}
    \end{IEEEeqnarray}
for all $k=0,\dots,K-1$.  The initial condition of the nominal state vector translates to the initial conditions $\delta x_0=\Delta x_0=0$.  From equation \eqref{eq:error_evolution}, we have $\delta x_1-\Delta x_1=0$ or $\delta x_1=\Delta x_1$.  It follows by induction that $\delta x_k=\Delta x_k$ for $k=0,\dots,K$.
\end{proof}

\edit{Because the error resulting from left-endpoint integration is well-known to be bounded in proportion to $\Delta t_k^2$ \cite{ascher2011first}, the above result can be extended to show that the solutions of the two methods above converge pointwise to one another as $\Delta t_k$ approaches zero.} It follows that either of the two expressions in equations \eqref{eq:discrete_lin} or \eqref{eq:lin_discrete} may be approximated with the other if $\Delta t_k$ and $T$ are sufficiently small.  Moreover, although equation \eqref{eq:lin_discrete} reduces to equation \eqref{eq:discrete_lin} when approximate integration is performed, this does not necessarily imply that equation \eqref{eq:lin_discrete} is more accurate than equation \eqref{eq:discrete_lin}.  We arrive at this conclusion with Taylor's multivariate theorem \cite{apostol1967calculus}.  In particular, the exact transition $x_k$ provided by equation \eqref{eq:transition} and the approximate transition $\overline x_k+\delta x_k$ in equation \eqref{eq:discrete_lin} agree up to and including first order terms in both the state and control perturbation variables. This is generally not true for the transition $\overline x_k+\Delta x_k$ provided by equation \eqref{eq:lin_discrete}.  \edit{During preliminary computations of the examples described in Section \ref{sec:examples}, we observe numerical inaccuracies caused by linearizing before discretizing.  For certain problems, an error tolerance $\epsilon$ that is several orders of magnitude smaller can be achieved when discretizing before linearizing in contrast to the reverse order of these operations.} Because of the limited accuracy of the latter method, we employ the discrete linear system in equation \eqref{eq:discrete_lin}.

\section{Iterative Quadratic Program} \label{sec:sqp}

In this section, we describe an algorithm for solving the OCP formulated in Section \ref{sec:robust_control} using a two-stage iterative quadratic programming approach. We outline the algorithms here and refer the interested reader to a recent study for details on the convergence of iterative quadratic programs for nonlinear dynamic systems \cite{vu2023iterative}.  The first algorithm determines a control function that steers the ensemble from the uniform initial state to within a specified error of the target state, as given by equation \eqref{eq:endpoint_equalities_term_relaxed}.  The second algorithm is then applied to gradually adjust the steering control function to minimize the control energy objective in equation \eqref{eq:ocp_objective} while fixing the initial and terminal states achieved in the first stage.

\edit{The number of equality constraints in equation \eqref{eq:discrete_lin}, for $k=0,\dots,K-1$, is equal to $n(N_{\alpha}+1)(N_{\beta}+1)K$, and this quantity ranges from tens to hundreds of thousands for the examples we consider in Section \ref{sec:examples}. } Such a large number of equality constraints could be problematic even for efficient quadratic programming packages.  Fortunately, the problem can be simplified by recursively evolving the dynamics according to
\begin{IEEEeqnarray}{lll} 
    \delta x_1&=&\bm A_0\delta x_0+\bm B_0 \delta u_0, \nonumber \\
    \delta x_2&=&\bm A_1\bm A_0\delta x_0+\bm A_1\bm B_0\delta u_0+\bm B_1\delta u_1,  \nonumber \\
    &\vdots& \label{eq:evolution_recursion} \\
    \delta x_K&=&\prod_{k=0}^{K-1}\bm A_k \delta x_0+\sum_{k=0}^{K-1}\left(\prod_{j=k+1}^{K-1}\bm A_j\right)\bm B_k\delta u_k,  \nonumber
\end{IEEEeqnarray}
where we define $\prod_{j=K}^{K-1}\bm A_j=I_{n(N_{\alpha}+1)(N_{\beta}+1)}$.  It follows from $\delta x_0=0$ that the zero-input response term $\prod_{k=0}^{j}\bm A_k \delta x_0$ vanishes from the above sequence of equations.  Because we are concerned with steering the terminal state of the system, the only equation from the above sequence that requires consideration is the one that defines $\delta x_K$ in terms of the control variables. We define the $n(N_{\alpha}+1)(N_{\beta}+1)\times mK$ evolution matrix
\begin{eqnarray}
    H=\left[\bm A_{K-1}  \cdots  \bm A_1\bm B_0,\dots,\bm A_{K-1}\bm B_{K-2},\; \bm B_{K-1}\right], \label{eq:H}
\end{eqnarray}
so that $\delta x_K=H\delta u$, where $\delta u =[\delta u_0',\dots,\delta u_{K-1}']'$.  We are now in position to present the control algorithms.

Consider a nominal control vector $\overline U=[\overline U_0',\dots,\overline U_{K-1}']'$ and the evolution of the associated state of the bilinear system $\overline x=[\overline x_0',\dots,\overline x_K']'$ in equation \eqref{eq:transition}.  These vectors are used to define or update the matrices in equations \eqref{eq:A}-\eqref{eq:B} and \eqref{eq:H}.  The control perturbation vector $\delta u$ that will move $x_K$ closer to $x_T$ is constrained according to
\begin{subequations} \label{eq:control_inequalities_discrete}
\begin{align} 
U_{\min}&\le \overline{U}_k + \delta u_k\le U_{\max}, \,\, \forall k=0,\ldots,K-1, \\ 
\Delta U_{\min}&\le  \frac{\overline{U}_{k+1}+\delta u_{k+1}-\overline{U}_{k}-\delta u_{k}}{\Delta t_k}\le \Delta U_{\max}, \nonumber \\ & \qquad \qquad \qquad \qquad\,\, \forall k=0,\ldots,K-2,
\end{align}
\end{subequations}
following the OCP constraints \eqref{eq:control_inequalities}, and is determined by solving the quadratic program defined by
\begin{equation} \label{eq:ocp_discrete_min_term}
\begin{array}{ll}
\text{min}_{\delta u}   &\edit{\|P\left(H\delta u+\overline x_K-x_T\right)\|^2 +\lambda \|\Lambda\delta u\|^2,} \\
\text{s.t.}   &\text{Inequality constraints in Eqns.  \eqref{eq:control_inequalities_discrete}},
 \end{array}
\end{equation}
where $\Lambda= \text{diag}(\Delta t_0,\dots,\Delta t_{K-1})\otimes I_m$ and  $\lambda$ is a regulation parameter that is adjusted between iterations.  The penalty term weighted by $\lambda$ in the objective function serves to regulate the norm of the perturbed control vector to render linearization applicable.  The solution $\delta u$ is used to update the control function $\overline U:=\overline U+\delta u$, with which the associated evolution of the bilinear state $\overline x$ is simulated according to equation \eqref{eq:transition}.  The procedure is repeated until $\|P(\overline x_K-x_T)\|\le \epsilon$ or until $\|\Lambda\delta u\|\le \delta$, where $\delta$ is a positive threshold.  When any one of these two metrics are achieved, the updated vectors $\overline U$ and $\overline x$ are stored and the steering algorithm is terminated.  \edit{While the initialization of the nominal control input may be specified arbitrarily, a judicious selection may promote better convergence.}
Moreover, the regularization parameter is adjusted between iterations according to $\lambda=\lambda_0\|P(\overline x_K-x_T)\|^2$, where $\lambda_0$ is a positive constant.

The minimum energy control function is computed as follows.  First, the vectors $\overline U$ and $\overline x$ that result from the first stage and the associated matrix $H$ are used to initialize the energy-minimizing algorithm.  The resulting quadratic program is
\begin{equation} \label{eq:ocp_discrete_min_energy}
\begin{array}{ll}
\text{min}_{\delta u}   & \|\Lambda\left(\overline U+\delta u\right)\|^2 +\mu \|\Lambda\delta u\|^2, \\
\text{s.t.}   & \edit{PH\delta u = 0,} \\
              &\text{Inequality constraints in Eqns.  \eqref{eq:control_inequalities_discrete}},
 \end{array}
\end{equation}
where $\mu$ serves the same purpose as $\lambda$ does in the first stage.  \edit{Because the target state may not be exactly reachable, the first constraint in equation \eqref{eq:ocp_discrete_min_energy} requires the terminal state of the minimal energy iterative algorithm to remain at the terminal state achieved with the steering controller.} The regulation parameter $\mu$ is updated between iterations according to $\mu:=0.9\mu$ if $\|\Lambda\delta u\|\le 2\delta$.  The solution $\delta u$ is used to define the updated control function $\overline U:=\overline U+\delta u$, from which the updated state $\overline x$ is simulated according to equation \eqref{eq:transition}.  If $\|\Lambda\delta u\|\le \delta$ at any stage of the iteration, then the minimal energy algorithm is terminated.  Otherwise, the matrix $H$ in equation \eqref{eq:H} is updated using the vectors $\overline U$ and $\overline x$, and the process is repeated. 

\section{Computational Studies}  \label{sec:examples}

\begin{figure*}[t]
    \centering
    \begin{subfigure}[t]{0.5\textwidth}
        \centering
        \includegraphics[width=1\linewidth]{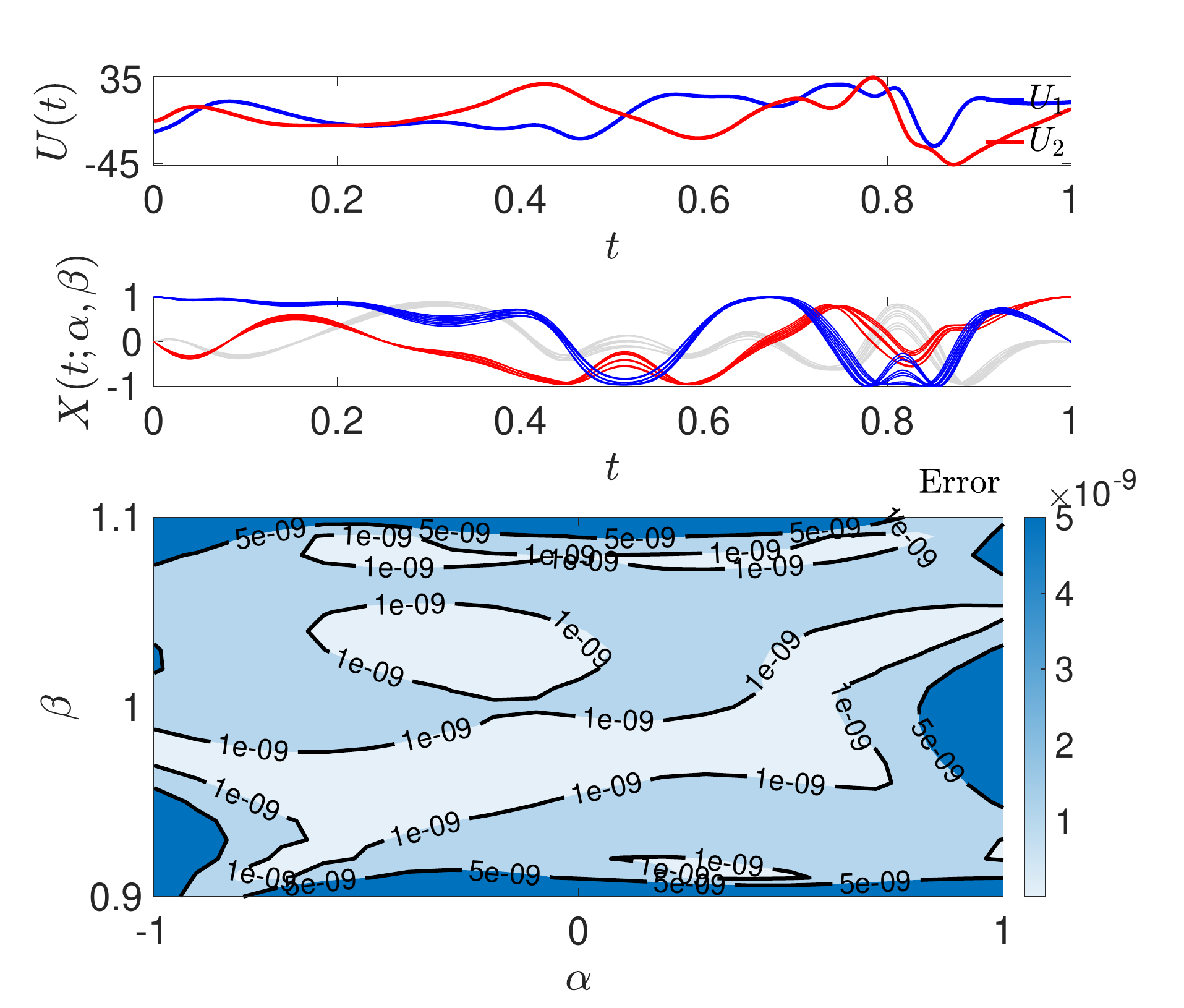}
    \end{subfigure}%
    ~ 
    \begin{subfigure}[t]{0.5\textwidth}
        \centering
        \includegraphics[width=1\linewidth]{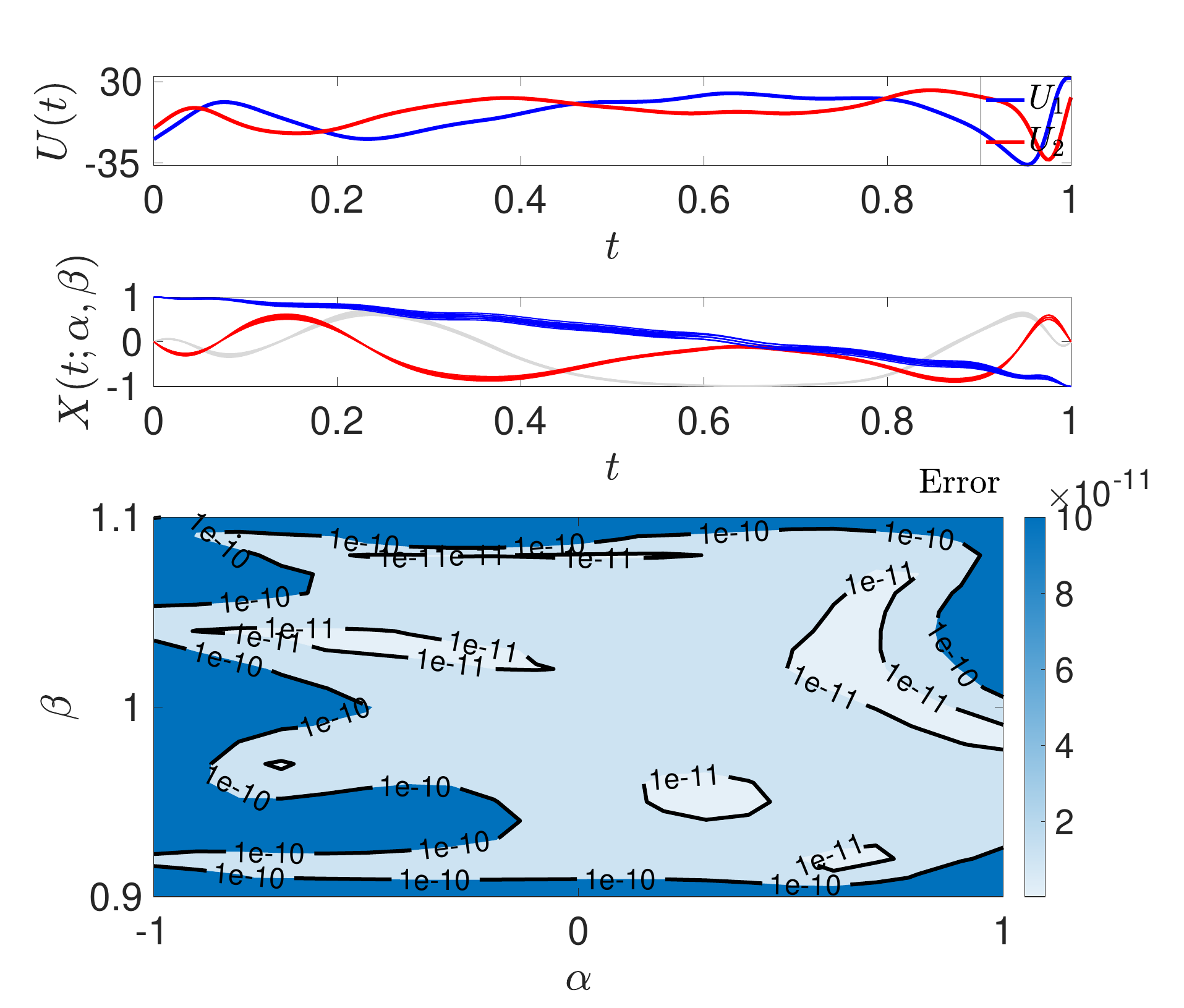}
    \end{subfigure}
\caption{\edit{State transfer of the Bloch system from $[0,0,1]'$ to $[1,0,0]'$ (left column) and from $[0,0,1]'$ to $[0,0,-1]'$ (right column).  From top to bottom are the controls, states, and the resulting terminal error contours obtained by repeated simulation over the parameter space.}}
\label{fig:bloch}
\end{figure*}

The performance of the iterative quadratic programming algorithm presented in Section \ref{sec:sqp} to solve the OCP described in Section \ref{sec:ocp} will be demonstrated for two examples that arise in quantum control applications.   The computations are performed in Matlab R2023a on a MacBook Pro with 32 GB of usable memory and an Apple M2 Max processing chip. The quadratic program at each stage of the iteration is implemented with the general-purpose Matlab function \verb+quadprog+ using the sparse-linear-convex algorithm and sparse linear algebra operations.  The CPU user load ranges between 12\% and 50\% of the maximum capability of the computer and the used memory is less than 3.5 GB.

\subsection{Nuclear Magnetic Resonance Spectroscopy} \label{sec:bloch}

Imaging modalities that take advantage of nuclear magnetic resonance (NMR) apply a strong constant magnetic field to a sample of nuclei and then apply radio-frequency (RF) fields in the transverse plane to manipulate the nuclear spins of the sample.  The spin dynamics are modeled using the Bloch equations \cite{pauly1991parameter,mabuchi2005principles, li2009ensemble}.  The system of Bloch equations in a rotating reference frame without relaxation \cite{torrey1956bloch, li2009ensemble} may be written as 
\begin{equation} \label{eq:bloch}
    \dot X = \begin{bmatrix}
        0 & -\alpha & \beta U_1 \\
        \alpha & 0 & -\beta U_2 \\
        -\beta U_1 & \beta U_2 & 0
    \end{bmatrix}X,
\end{equation}
where $X(t;\alpha,\beta)$ represents the bulk magnetization of the nuclei, and $U_1(t)$ and $U_2(t)$ represent the applied fields.  Variations in system parameters appear as dispersion in the intrinsic frequencies of the nuclei and the strength of the applied field \cite{levitt1986composite}.   Here, the ensemble is defined by the continuum of parameter values $\alpha\in [-1,1]$ and $\beta\in [0.9,1.1]$, which respectively represent variations in Larmor frequency and the amplitude of the applied field.  We seek controls $U_1(t)$ and $U_2(t)$ of minimal energy that steer the ensemble state from the zero-input equilibrium state $X_0=[0,0,1]'$ to the excited state $X_T=[1,0,0]'$.   \edit{This example has gained significant interest \cite{li2011optimal, zlotnik2012iterative, ning2022nmr} and will be referred to as example (a) in the following. We also present results for another example (b) in which the ensemble state is steered from one marginally stable equilibrium state, $X_0=[0,0,1]'$, to another given by $X_T=[0,0,-1]'$. 

Figure \ref{fig:bloch} displays the control functions, state vectors, and contours of the terminal error obtained by repeated simulation over the design region of parameters $\alpha$ and $\beta$ for example (a) in the left column and example (b) in the right column.  In addition to the parameters above, we use $T=1$, $K=300$, $N_{\alpha}=4$, $N_{\beta}=3$, and $\lambda_0=0.01$.  The total number of equations in  \eqref{eq:discrete_lin} is $n(N_{\alpha}+1)(N_{\beta}+1)K=18, 000$ and the size of $H$ in equation \eqref{eq:H} is $n(N_{\alpha}+1)(N_{\beta}+1)\times m(K-1)=60\times 598$.  Each problem terminates in about 3.3 minutes after reaching the maximum number of allowed iterations, which is specified at 800 for these examples.  It is evident from Figure \ref{fig:bloch} that the control algorithm successfully solves both problems with unprecedented error margins \cite{zlotnik2012iterative, ning2022nmr}.  Moreover, the controller is capable of achieving slightly smaller error margins, with smaller pulse amplitudes, for example (b) than example (a).}  

\subsection{Matter-Wave Splitting for Atom Interferometry} \label{sec:bec}

\begin{figure*}[t]
    \centering
    \begin{subfigure}[t]{0.5\textwidth}
        \centering
        \includegraphics[width=1\linewidth]{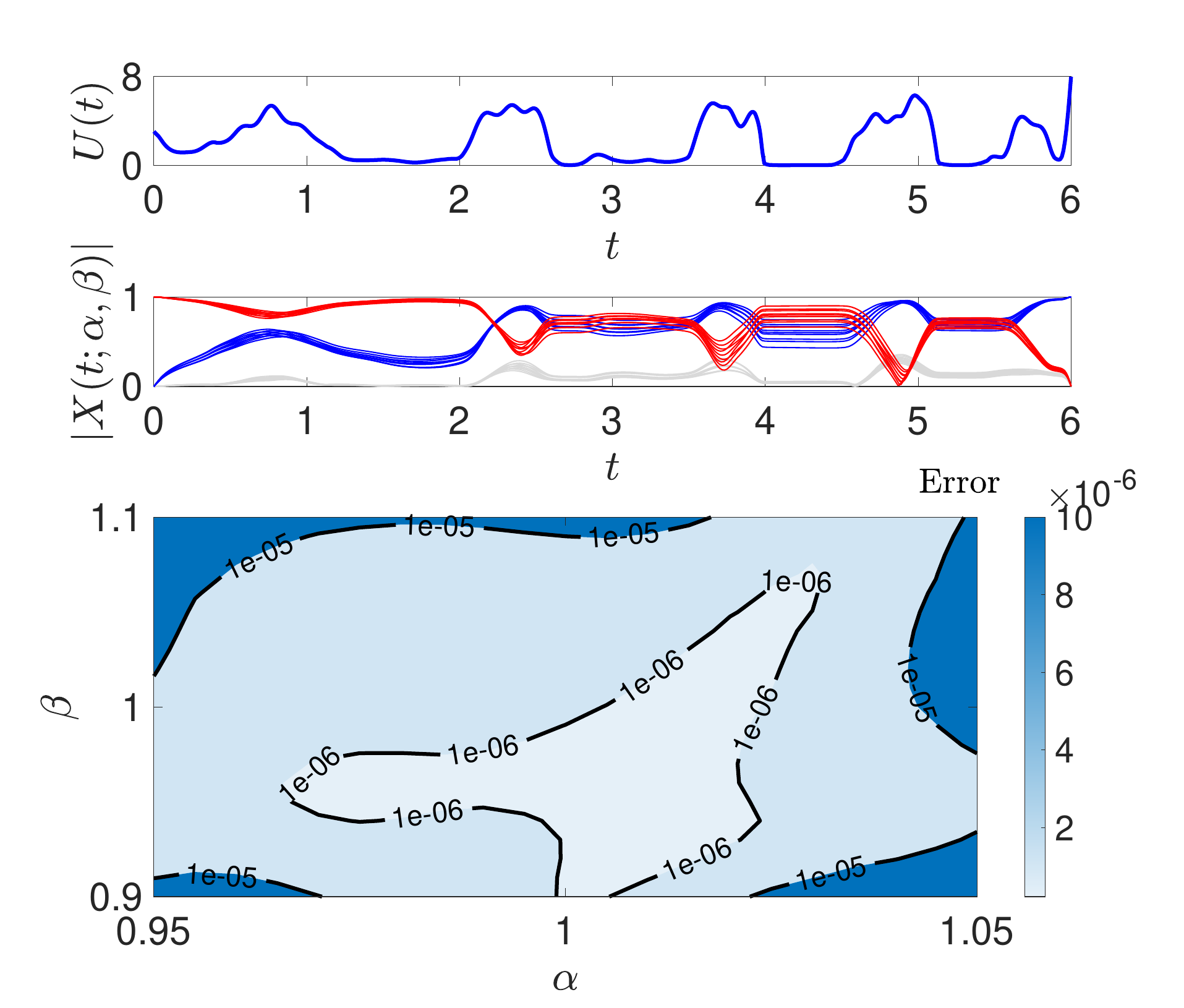}
    \end{subfigure}%
    ~ 
    \begin{subfigure}[t]{0.5\textwidth}
        \centering
        \includegraphics[width=1\linewidth]{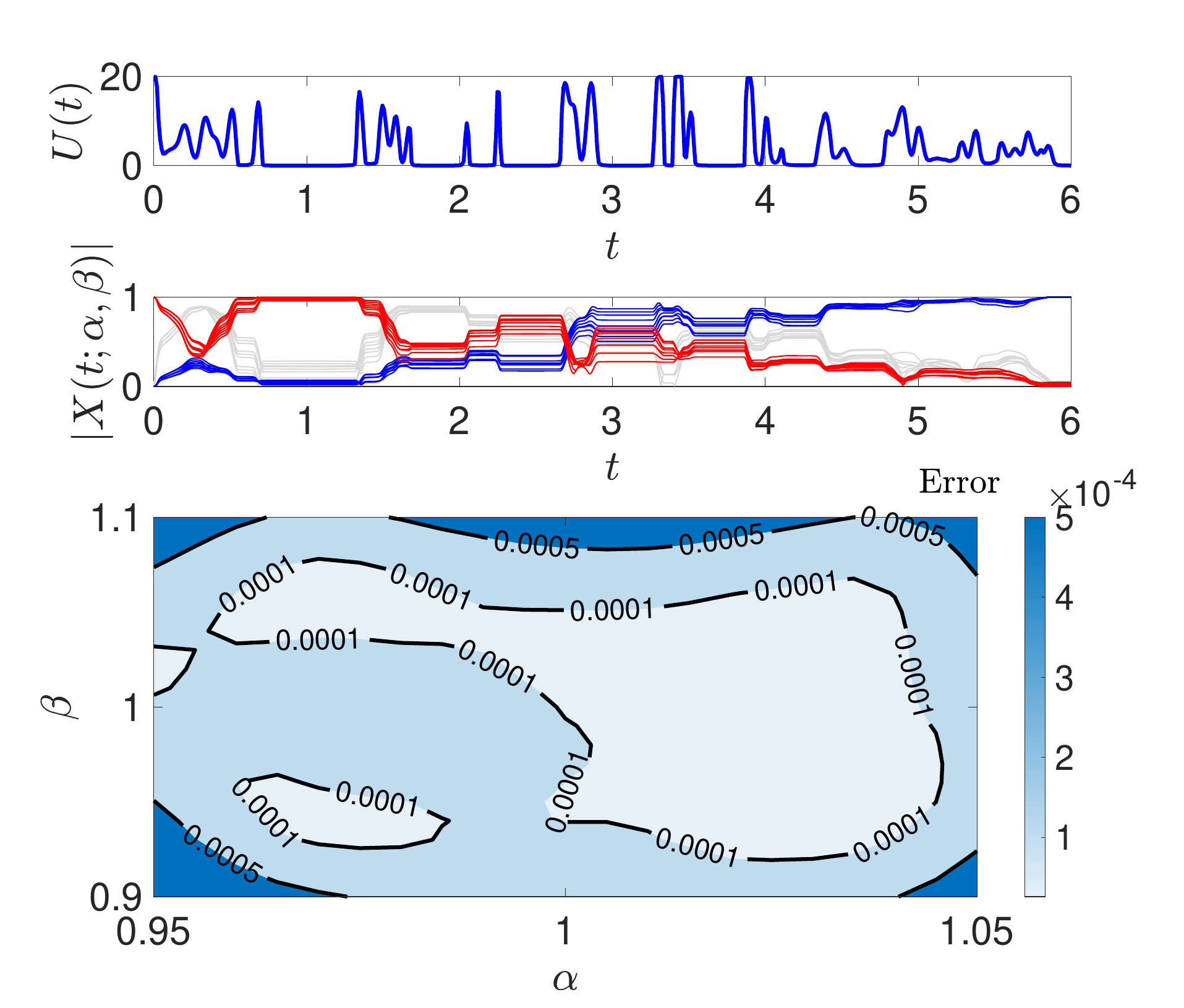}
    \end{subfigure}
\caption{\edit{State transfer of the Raman-Nath system from $|0\hbar k\rangle$ to $|\pm 2\hbar k\rangle$ (left column) and from $|0\hbar k\rangle$ to $|\pm 4\hbar k\rangle$ (right column). Absolute values correspond to norms of complex numbers.}}
\label{fig:split}
\end{figure*}

We consider a quantum control setting related to a proposed technique for interferometry, in which a dilute Bose-Einstein condensate (BEC) composed of atoms that are initially at rest is manipulated to elicit a diffraction pattern.  The relevant dynamics for the initial matter-wave splitting are modeled using the Raman-Nath equations \cite{wu2005splitting}.  Standing-wave optical pulses modulated by rectangular \cite{edwards2010momentum}, Gaussian \cite{muller2008atom}, and other transcendental envelopes \cite{cassidy2021improved} were designed to split the stationary condensate into a definite state or a superposition of high-order momentum states.  \edit{Recently, robust optimal control was applied to the Raman-Nath equations including compensation for 10\% uncertainty in the optical pulse intensity  \cite{delima2023optimal}.  We extend the results to compensate for 10\% uncertainty in light intensity together with 5\% uncertainty in photon recoil energy \cite{campbell2005photon}.}

The wave function $\Psi(t,x)$ is governed by the one-dimensional Schr{\"o}dinger equation \cite{wu2005splitting, cassidy2021improved, delima2023optimal},
\begin{equation}
    i \frac{\partial \Psi}{\partial t} = -\frac{\hbar}{2 m}\frac{\partial^2\Psi}{\partial x^2} +U(t)\cos^2(2k_0x)\Psi,
\end{equation}
where $U(t)$ is the amplitude of the light shift potential and $k_0$ is the vacuum wave number of the photons. In this example, we adhere to convention and use the symbols $i$, $m$, and $x$ to denote the imaginary unit, the mass of the BEC, and the independent spatial variable, respectively. \edit{As in prior studies \cite{wu2005splitting, cassidy2021improved}, we expand the wavefunction and write $\Psi(t,x)=\sum_{n}\int d\bm k C_{2n}(t;\bm k)e^{i(2nk_0+k)x}$,  where $k$ represents the wave number distribution.  The complex-valued coefficient $C_{2n}$ represents the probability amplitude of measuring the momentum state $| 2n\hbar k\rangle$.
By adopting common physical assumptions \cite{delima2023optimal,cassidy2021improved},} the dynamics of diffraction may be approximated with the vector $\mathcal X=[C_0,C_2,\dots,C_{2N'}]'$ of nonnegative momentum coefficients whose dynamics are given by
\begin{equation} \label{eq:bec}
    \dot{\mathcal X} =- i \left(\alpha \mathcal A_0+ U(t)\beta \mathcal B_0\right) \mathcal X,
\end{equation}
where the $(N'+1)\times (N'+1)$ matrices are defined by $\mathcal A_0= \omega_r\text{diag}(0,4,\dots,(2N')^2)$ and
\begin{IEEEeqnarray}{lll}
      \mathcal B_0= \frac{1}{2}\begin{bmatrix}
        0 & \sqrt{2} &  &     \\
       \sqrt{2}  & 0 & 1 &   \\
         & 1 & \ddots & 1   \\
         & & 1&  0 
    \end{bmatrix}. \label{eq:B0_matrix} 
\end{IEEEeqnarray}
\edit{We include the parameters $\alpha\in [0.95,1.05]$ and $\beta \in [0.9,1.1]$ to compensate for 5\% and 10\% uncertainty in photon recoil energy and light intensity, respectively.}
Finally, we expand the complex-valued state vector into its real and imaginary components and substitute the expression $\mathcal X=\text{Re}(\mathcal X)+ i\text{Im}(\mathcal X)$ into equation \eqref{eq:bec}.   By equating real and imaginary parts and defining $X=[\text{Re}(\mathcal X)',\text{Im}(\mathcal X)']'$, the equivalent real-valued bilinear ensemble system is
\begin{equation*}
\dot X(t)=\alpha \begin{bmatrix}
        0 & \mathcal A_0 \\
        -\mathcal A_0 & 0
    \end{bmatrix} X(t)+U(t)\beta \begin{bmatrix}
        0 & \mathcal B_0 \\
        -\mathcal B_0 & 0
    \end{bmatrix}  X(t).
\end{equation*}
The initial and desired target states are defined by $X_0=[1,0,\dots,0]'$ and  $X_T=[0,\dots,0,1,0,\dots,0]'$, where the only nonzero component of the target state appears in the $(n'+1)$-th entry. Here, $n'$ is an integer representative of the target momentum state $|\pm 2n'\hbar k\rangle$.

\edit{Figure \ref{fig:split} shows the control functions, state vectors, and error contours for $n'=1$ in the left column and $n'=2$ in the right column.  The BEC system is truncated at $N'=7$ for which the dimension of the ensemble state vector $X(t,\alpha,\beta)$ is $n=2(N'+1)=16$.  The other parameters used for the computation are $T=6$, $K=600$, $N_{\alpha}=6$, $N_{\beta}=3$, $\lambda_0=0.01$, $U_{\min}=0$, and $U_{\max}=10n'$.  For both matter-wave splitting examples, the total number of equations in \eqref{eq:discrete_lin} is $n(N_{\alpha}+1)(N_{\beta}+1)K=268,800$ and the size of $H$ in equation \eqref{eq:H} is $n(N_{\alpha}+1)(N_{\beta}+1)\times (K-1)=448\times 599$.  For each problem, the algorithm converges to a minimum energy control function that satisfies the imposed bounds in about 80 quadratic programming iterations during a total 20 minutes of computation time.  As for the NMR examples, Figure \ref{fig:split} demonstrates that the control function is robust over the design region of parameter values for matter-wave splitting, although the terminal error is in general greater because of greater complexity of the system dynamics.}

\section{Conclusion} \label{sec:conc}
\vspace{-.5ex}

We have designed a computational method for open-loop minimum-norm control synthesis for fixed-endpoint transfer of bilinear ensemble systems that are indexed by two continuously varying parameters.  The ensemble state is approximated using a truncated basis of Legendre polynomials.  The dynamics are linearized at each stage of the iteration about control and state trajectories to formulate a sequence of quadratic programs for computing perturbations to the control that successively improve the objective until convergence.  We show that the approximation quality depends on the order in which linearization and exact discretization are performed.  In particular, we prove that the two orders of operations result in different systems that are approximately equivalent in the sense of numerical quadrature.

\edit{The developed two-stage interative quadratic programming algorithm for solving the formulated class of optimal control problems is demonstrated for the Bloch system that appears in nuclear magnetic resonance, as well as the Raman-Nath equations that appear in the beamsplitter process of atom interferometry.  For both magnetic resonance and atom interferometry, the control algorithm successfully converges to robust pulse designs that achieve the desired transfer of states with unprecedented fidelity over the specified uncertain parameter space.  Although the computation of trajectories and linear system approximations are performed without symbolic algebra, the matrices and time evolution are currently updated with a for-loop at each stage of the iteration.  This is generally not scalable in Matlab to even more complex systems that may exceed millions of constraints.  Future work can extend the iterative quadratic programming approach presented here to offer more computationally expedient and tractable formulations that involve purely matrix-vector operations, and may benefit from the use of high performance computing.}

\vspace{-1ex}
\linespread{1}
\bibliographystyle{unsrt}
\bibliography{references.bib}

\end{document}